\documentclass[12pt,thmsa, reqno]{amsart}

\usepackage{mathrsfs}
\usepackage{amsfonts}
\usepackage{amssymb}
\usepackage{amsmath}

\usepackage[dvips]{graphics}
\input{amssym.def}
\input{amssym.tex}

\def\gnk{G_{n,k}}
\def\agnk{ G(n,k)}
\def\gnkp{G_{n,k'}}
\def\agnkp{ G(n,k')}

\def\rk{\bbr^k}
\def\rkp{\bbr^{k'}}

\def\iagr{\intl_{\agnk}}
\def\iagrp{\intl_{\agnkp}}

\def\Cal{\mathcal}

\def\M{{\Cal M}}

\def\gnk{G_{n,k}}
\def\gnkp{G_{n,k'}}

\def\bbr{{\Bbb R}}

\def\bbc{{\Bbb C}}

\def\bbs{{\Bbb S}}

\def\const{{\hbox{\rm const}}}

\def\span{{\hbox{\rm span}}}

\def\Pr{{\hbox{\rm Pr}}}

\def\gnk{G_{n,k}}
\def\gnkp{G_{n,k'}}

\def\q{\quad}
\def\qq{\qquad}

\def\rn{\bbr^n}

\def\rk{\bbr^k}

\def\part{\partial}
\def\intl{\int\limits}

\def\Gam{\Gamma}

\def\a{\alpha}
\def\om{\omega}

\def\vp{\varphi}

\def\g{\gamma}
\def\gam{\gamma}

\def\sig{\sigma}

\def\z{\zeta}
\def\e{\varepsilon}
\def\t{\tau}
\def\th{\theta}

\def\msL{\mathscr{L}}
\def\msH{\mathscr{H}}

\newtheorem{theorem}{Theorem}[section]
\newtheorem{lemma}[theorem]{Lemma}

\theoremstyle{definition}
\newtheorem{definition}[theorem]{Definition}
\newtheorem{example}[theorem]{Example}

\theoremstyle{remark}
\newtheorem{remark}[theorem]{Remark}

\theoremstyle{corollary}

\newtheorem{proposition}[theorem]{Proposition}

\numberwithin{equation}{section}

\newcommand{\be}{\begin{equation}}
\newcommand{\ee}{\end{equation}}

\newcommand{\bea}{\begin{eqnarray}}
\newcommand{\eea}{\end{eqnarray}}
\newcommand{\Bea}{\begin{eqnarray*}}
\newcommand{\Eea}{\end{eqnarray*}}

\def\sideremark#1{\ifvmode\leavevmode\fi\vadjust{\vbox to0pt{\vss
 \hbox to 0pt{\hskip\hsize\hskip1em
\vbox{\hsize2cm\tiny\raggedright\pretolerance10000
 \noindent #1\hfill}\hss}\vbox to8pt{\vfil}\vss}}}%

\begin{document}

\title[On Radon Transforms ]
{On Radon Transforms Between Lines and Hyperplanes}

\author{Boris Rubin and Yingzhan Wang*}
\thanks{* Supported by the  National Natural Science Foundation of
China under Grant \#11271091 and  the China Scholarship Council under Grant \#201308440083. }
\address{
Department of Mathematics, Louisiana State University, Baton Rouge,
LA, 70803 USA}
\email{borisr@math.lsu.edu}
\address{School of Mathematics and Information Science,
Guangzhou University, Guangzhou 510006, China}
\email{wyzde@gzhu.edu.cn}





\begin{abstract}
We obtain new inversion formulas for the Radon transform and its dual between lines and hyperplanes in $\rn$. The Radon transform in this setting is non-injective and the consideration is restricted to the so-called quasi-radial functions that are constant on symmetric clusters of lines. For the corresponding  dual  transform, which is injective, explicit inversion formulas are obtained  both in the  symmetric case and in  full generality. The main tools are  the  Funk transform on the sphere, the Radon-John $d$-plane transform in $\rn$, the Grassmannian modification of the Kelvin  transform,  and the Erd\'elyi-Kober fractional integrals.
\end{abstract}

\maketitle

\section{Introduction}

Let $\msL$ and $\msH$ be  the manifolds  of all non-oriented lines $\ell$ and all non-oriented hyperplanes $h$ in $\bbr^n$, respectively.
In the present article we consider the Radon-like transform that takes  functions  on $\msL$ to  functions on $\msH$. We also consider  the
 corresponding dual transform acting in the opposite direction. Both transforms   are  defined by the  integrals
\be\label{rads} (Rf)(h)=\intl_{\ell\subset h}f(\ell)\,d_h \ell,\qquad (R^*\vp)(\ell)= \intl_{h \supset \ell} \vp(h)\,d_\ell h,\ee
where the integration is performed with respect to  the corresponding canonical measures. Our aim  is to obtain  explicit inversion formulas for $R$ and $R^*$ in the cases when these operators are injective.

The manifolds $\msL$ and $\msH$ are important representatives of the more general vector bundles $G(n,k)$ over the Grassmann manifolds $G_{n,k}$ of $k$-dimensional linear subspaces of $\rn$. Elements of $G(n,k)$ are non-oriented $k$-dimensional affine planes in $\rn$.
 The corresponding  Radon transforms $R_{k,k'}$, $R^*_{k,k'}$ that take functions on $G(n,k)$ to functions on $G(n,k')$, $1\le k<k'\le n-1$, and backwards, were considered by   Gonzalez and   Kakehi  \cite{GK03, GK04} who studied these operators on smooth functions in the group-theoretic terms. The paper   \cite{GK03}  contains an explicit inversion formula  for
 $R_{k,k'}$ in the case of  $k'-k$ even. This formula was obtained by applying the Fourier transform over fibers and using the corresponding inversion formula for   compact Grassmannians due to Kakehi \cite{Ka99}. We also mention the paper \cite{G01} by  Gonzalez that
  contains   the  range description  of  the plane-to-line transform for smooth functions on $\bbr^3$.

 Inversion formulas of different kind for both $R_{k,k'}$ and  $R^*_{k,k'}$ in Lebesgue spaces were obtained by the first-named co-author \cite{Ru04}  who reduced the problem to the compact case with the aid of a certain analogue of the stereographic projection. The dimensions $k$ and $k'$ in \cite{Ru04} can be of arbitrary parity and the corresponding inversion formula for the compact Grassmannians was borrowed from Grinberg and Rubin \cite{GR04}.

 One should also mention the work by Strichartz \cite{Str86}, who developed harmonic analysis on Grassmannian bundles, and a series of publications
 related to  integral geometric problems on compact Grassmannians; see, e.g.,  \cite{Al, AGS, GR04, Ka99, Ru13, SZ, Zha1, Zha2} and references therein. The boundedness of the operators $R_{k,k'}$ and  $R^*_{k,k'}$ in  $L^p$ spaces with power weights was studied in \cite{Ru14}.

It is a challenging open problem to find an alternative approach to inversion formulas for  $R_{k,k'}$ and $R^*_{k,k'}$ that would be straightforward (without  stereographic projection), available for all admissible $k$ and $k'$, and applicable to large classes of functions (not only to $C^\infty$ and rapidly decreasing ones).
 Here one should take into account that for $1\le k<k'\le n-1$, the operators $R_{k,k'}$ and $R^*_{k,k'}$ are injective if and only if $k+k'\le n-1$ and $k+k'\ge n-1$, respectively; see \cite[Propositions 1.3 and 1.4]{Ru04} for precise statements. Because $\dim G(n,k)=(n-k)(k+1)$, the condition $k+k'\le n-1$ is equivalent to $\dim G(n,k)\le \dim G(n,k')$.

In the present article we give a solution to the aforementioned inversion problem in the  case $1=k<k'=n-1$. This simple case  is less technical and  reflects basic features of  Radon transforms on Grassmannian bundles.

The paper is organized as follows. Section 2 contains necessary preliminaries.  It is worth noting that planes in $\rn$ can be parametrized in different ways and our main concern in this section is to choose  suitable parametrizations.

Section 3 is devoted to the operator $R\equiv R_{1,n-1}$. Because \[\dim G(n,1)=2n-2>n= \dim G(n,n-1),\] this operator  is non-injective. The situation changes if we restrict $R$ to functions $f$ satisfying additional symmetry. Specifically, we assume $f$ to be constant on symmetric clusters of parallel lines in every direction. We call such functions  {\it quasi-radial}. It will be shown that if $f$ is quasi-radial, then $Rf$ is  a tensor product of the classical Funk transform on the sphere \cite{GGG, H11, Ru15}
and a one-dimensional fractional integration operator of the Erd\'elyi-Kober type \cite{Ru15}. Both can be explicitly inverted. The main result of this section is presented by Theorem \ref{kuku05a}.

In Section 4 we obtain explicit inversion formulas for the dual transform $R^*\equiv R^*_{1,n-1}$. Here we consider  three different approaches.
The first one deals with even locally integrable   functions $\vp$ on $\msH$ such that $\vp (h)=\vp (-h)$ for all $h \in \msH$, and relies on  averaging over the line clusters. The second approach covers the general case  and invokes a Grassmannian modification of the Kelvin transform introduced in \cite{Ru04}.  Here $\vp$ is assumed to be continuous with prescribed behavior at the origin and at infinity or belonging to some weighted $L^p$-space. The third approach   is based on a certain alternative parametrization of line and hyperplanes. Main results of this section are presented by Theorems \ref{kut},   \ref{ci}, \ref{too1}, and \ref {ffi}.

The methods of the present article can be generalized to the Radon transforms $R_{k,k'}$ and $R^*_{k,k'}$ for arbitrary $0<k<k'<n$. We plan to consider this case in the forthcoming publication.

{\bf Acknowledgements.} The main part of the work was done when the second-named author was visiting Louisiana State University in 2014-2015. Both authors are grateful to the  China Scholarship Council  for the support and the administration of the  Department of Mathematics at Louisiana State University for the hospitality.

 \section{Preliminaries}

\setcounter{equation}{0}

\subsection{Notation.}   We recall that $\agnk$  is the
affine Grassmann manifold  of  non-oriented $k$-dimensional
planes  in $\rn$, $0 < k <n$;  $\gnk$  is the compact
Grassmann manifold  of  $k$-dimensional linear subspaces  of $\rn$. If $\xi\in \gnk$, then  $\xi^\perp $ denotes the
orthogonal complement of $\xi $ in $\rn$.  Each element
  $\t$ of  $\agnk$ is  parameterized by the pair
$(\xi, u)$, where $\xi \in \gnk$ and $ u \in \xi^\perp$. In the following  $ |\t|$ denotes the
 Euclidean distance from the origin to $\t  \equiv\t(\xi, u)$ so that
 $|\t|=|u|$ is the Euclidean norm of $u$. We write $C_0 (\agnk)$ for the space of all continuous functions $f$ on $\agnk$  satisfying $\lim\limits_{|\t| \to \infty} f(\t)=0$. We also set
  \[C_\mu (G(n,k))=\{ f\in   C(G(n,k)) : \, f(\t)=O(|\t|^{-\mu})\}, \]
  \[C_\mu (\rn)=\{ f\in C(\rn): f(x)=O(|x|^{-\mu}),\] where $C(G(n,k))$ and $C(\rn)$ are  the space of continuous functions on $G(n,k)$ and $\rn$, respectively.

The manifold
$\agnk$ will be endowed with the product measure $d\t=d\xi du$,
where $d\xi$ is the
 $O(n)$-invariant probability measure  on $\gnk$ and $du$ is the usual volume element on $\xi^\perp$.

For $k'>k$ and $\eta \in \gnkp$, we denote by  $G_k (\eta)$ the
Grassmann manifold of all $k$-dimensional linear subspaces of
$\eta$.  In the following, $\bbs^{n-1} = \{ x \in \bbr^n: \ |x| =1 \}$ is the unit sphere
in $\bbr^n$.   For $\theta \in \bbs^{n-1}$, $\{\th\}$ denotes the one-dimensional linear subspace spanned by $\th$,
$d\theta$ stands for  the surface element on $\bbs^{n-1}$, and  $\sigma_{n-1} =  2\pi^{n/2} \big/ \Gamma (n/2)$ is the surface area of $\bbs^{n-1}$.  We set $d_*\theta= d\theta/\sigma_{n-1}$ for  the normalized surface element on $\bbs^{n-1}$.

We write $ e_1, \ldots, e_n$ for the coordinate unit
 vectors in $\rn$. Given $1\le k<k'<n$,  the following notations are used for the coordinate planes:
 \[ \rk\!=\!\bbr e_1 \oplus \ldots \oplus\bbr e_k, \quad
 \rkp\!=\!\bbr e_1 \oplus \ldots \oplus \bbr e_{k'}, \quad
  \bbr^{n-k}\!=\!\bbr e_{k+ 1}\oplus \ldots \oplus \bbr
 e_{n}.\]

 \subsection{Radon Transforms on Affine Grassmannians} \label{rafgr}

  Let $\agnk $ and $ \agnkp$ be a pair of affine Grassmann manifolds of non-oriented
 $k$-planes $\t$ and $k'$-planes $\z$ in $\rn$,
 respectively; $ \; 1\le k < k' \le n-1$. We write
 \be \label{para1}\t \! \equiv \! \t(\xi, u)\in \agnk,  \qquad
 \z \!  \equiv \! \z(\eta, v) \in \agnkp,  \ee
where $\xi\in \gnk$, $u \in \xi^\perp$, $\eta \in  \gnkp$, $  v \in \eta^\perp$.
The Radon transform  of a function $f(\t) $ on $\agnk$ is a
function $(R_{k,k'}f)(\z)$  on $\agnkp$ defined by
\be \label {rfold} (R_{k,k'}f)(\z)  = \intl_{\t
\subset \zeta} f(\t)\, d_\z \t= \intl_{\xi \subset \eta} d_\eta \xi
\intl_{\xi^\perp \cap  \eta} f(\xi, v+x) \,dx. \ee Here $d_\eta \xi$
denotes the probability measure on the Grassmannian $G_k (\eta)$.
This transform
integrates   $f(\t) $ over all $k$-planes $\t$   in  the
$k'$-plane $\z$.

The dual Radon transform  of a function $\vp(\z)$ on $\agnkp$ is a function $(R^*_{k,k'}\vp)(\t)\equiv (R^*_{k,k'}\vp)(\xi, u)$  on $\agnk$ defined by \bea (R^*_{k,k'}\vp)(\t)
&=& \intl_{\z \supset \t} \vp(\z)\, d_\t \z= \intl_{\eta  \supset \xi}
\vp(\eta +u) \,d_\xi \eta\nonumber\\
\label{dut}&=&\intl_{\eta \supset \xi}
\vp(\eta,\Pr_{\eta^\perp} u)\, d_\xi \eta. \eea Here
$\Pr_{\eta^\perp} u$ is the orthogonal projection of $u \;
(\in \xi^\perp)$ onto $\eta^\perp (\subset \xi^\perp)$, $d_\xi
\eta$ is the relevant probability measure.  This transform
integrates  $\vp(\z)$ over all $k'$-planes $\z$   containing  the
$k$-plane $\t$. In order to give (\ref{dut}) precise meaning, we choose
an orthogonal transformation  $g_\xi \in O(n)$ so that $g_\xi \rk=\xi$, and let
$O(n-k)$ be the subgroup of $O(n)$ that consists of orthogonal transformations preserving the coordinate plane
$\bbr^{n-k}$. Then (\ref{dut}) means \be (R^*_{k,k'}\vp)(\t) \equiv (R^*_{k,k'}\vp)(\xi, u)=
\intl_{O(n-k)} \vp(g_\xi \rho \rkp +u) \, d\rho. \ee

\begin{proposition} \label{prop1}\cite[Lemma 2.1]{Ru04} The equality
\be \iagrp (R_{k,k'}f)(\z) \,\vp(\z) \, d \z=\iagr f(\t)\, (R^*_{k,k'}\vp)(\t)\, d\t \ee holds provided  that at least one of these integrals exists in the Lebesgue sense
(i.e., it is finite  if $f$ and $\vp$ are replaced by $|f|$ and $|\vp|$,
respectively).
\end{proposition}
\begin{proposition} \label{prop2} ${}$ \hfill

{\rm (i)} If $f \in L^p (\agnk), \; 1 \le p < (n-k)/(k' -k)$,
then $(R_{k,k'}f)(\z)$ is finite for almost all $\z \in \agnkp$. If $ f\in C_\mu (\agnk)$, $\mu >
 k'-k$, then $(R_{k,k'}f)(\z)$ is finite for all $\z \in
 \agnkp$. The conditions $p < (n-k)/(k' -k)$ and $\mu >
 k'-k$ are sharp.

{\rm (ii)}  The dual transform  $(R^*_{k,k'}\vp)(\t)$ is finite a.e. on $\agnk$ for every locally integrable function $\vp$.
\end{proposition}

The statement (i) is proved in   \cite[Corollary 2.6]{Ru04}. The statement (ii) follows from the equality
\be \intl_{|\t|<a} (R^*_{k,k'}\vp)(\t) \, d\t =
 \const \intl_{|\z|<a}\vp (\z) \,(a^2 \! - \! |\z|^2)^{(k'-k)/2} \,d \z  \ee
which is a particular case of the formula (2.19) from  \cite{Ru04}.

\begin{proposition} \label{prop3}\cite[Lemma 2.3]{Ru04} For $\t \in \agnk$ and $\z \in \agnkp$,
let $ r=|\t|, \, s=|\z|$. If $f(\t)=f_0 (r)$ and $\vp(\z)= \vp_0
(s)$, then
\be\label{lmos1} (R_{k,k'}f)(\z)= \sig_{k'-k-1}\intl_s^\infty f_0(r) (r^2 -s^2)^{(k'-k)/2
-1} r dr, \ee
\be\label{lmos2} (R^*_{k,k'}\vp)(\t)= \frac{\sig_{k'-k-1} \,
\sig_{n-k'-1}}{\sig_{n-k-1} \, r^{n-k-2}}\intl_0^r \vp_0 (s) (r^2
-s^2)^{(k'-k)/2 -1} s^{n-k'-1} ds, \ee provided that the  corresponding
integrals exist in the Lebesgue sense.
\end{proposition}

\subsection{Alternative Parametrization of Lines and Hyperplanes}

In this subsection we consider the case when the Radon transform $R$ takes a function $f$ on $\msL=G(n,1)$ to a function  $Rf$ on $\msH=G(n,n-1)$.
For the future purposes we parametrize the manifolds  $\msL$ and $\msH$ in a slightly different way in comparison with  subsection \ref{rafgr}. Specifically, let
\be \tilde \msL=\{(\om, u): \; \om \in \bbs^{n-1}, u\in \rn, u\perp \om\}.\ee
Every line $\ell$ has the form  $\ell=\{\om\} +u$, where $(\om, u)\in \tilde \msL$, $\{\om\}=\span (\om)$. Setting $\ell=\ell (\om, u)$,
every function  $f$ on $\msL$  can be regarded as  a function $\tilde f(\om, u)=f(\{\om\} +u)$ on $\tilde \msL$ satisfying $\tilde f(\om, u)=\tilde f(-\om, u)$. We equip  $\tilde \msL$ with the product measure $d_*\om du $, where $d_*\om$ is the normalized surface measure on $\bbs^{n-1}$ and $du$ is the Euclidean volume element on $\om^\perp$. Then, for $f\in L^1 (\msL)$,
  \be\label{kuk}\intl_{\msL} f (\ell)\, d\ell=\intl_{\tilde \msL}  \tilde f(\om, u)\, d_*\om du=\intl_{\bbs^{n-1}}d_*\om \intl_{\om^\perp} \tilde f(\om, u)\, du,\ee
where the integral on the left-hand side has the same meaning as in subsection \ref{rafgr}.

For the hyperplane case  $h\in \msH=G(n,n-1)$, in parallel with the parametrization $h=h(\eta,v)$, where $\eta \in G_{n,n-1}$ and $v\in \eta^\perp$ (cf. (\ref{para1})), we set
$h=h(\th, t)=\{x\in \rn: x\cdot \th =t\}$, where $\th \in\bbs^{n-1}, t\in \bbr$. Let
\[\tilde \msH=\{(\th, t): \th \in\bbs^{n-1}, t\in \bbr\}.\]
Every function $\vp$ on $\msH$  can be thought of as a function $\tilde \vp(\th, t)=\vp(\th^\perp, t\th)$    on the cylinder $\tilde \msH$, so that
 $\tilde\vp(\th, t)=\tilde\vp(-\th, -t)$. Endowing $\tilde \msH$ with the measure $d_*\th dt$, we get
 \bea\intl_{\msH} \vp (h)\, dh&\equiv& \intl_{G_{n,n-1}} d\eta \intl_{\eta^\perp} \vp(\eta,v)\, dv=\intl_{O(n)} d\gam \intl_{\bbr} \vp(\gam e_n^\perp, t\gam e_n)\, dt \nonumber\\
&=&\label{kuk1}  \intl_{\bbs^{n-1}} d_*\th \intl_{\bbr} \vp(\th^\perp, t\th)\, dt=\intl_{\tilde \msH} \tilde \vp(\th, t)\,d_*\th dt.\eea
Abusing notation,  we identify
\be\label{kuku}
 \msL \equiv \tilde \msL, \qq  \msH \equiv \tilde \msH, \qq f\equiv\tilde f, \qq \vp\equiv\tilde \vp.\ee
According to (\ref{rfold}) and the identification (\ref{kuku}), the Radon transform (\ref{rfold}) can be written
  in the new parametrization as
\be\label{kuku01}
(Rf)(\th, t)=\intl_{\bbs^{n-1}\cap \th^\perp} d_\th \om\intl_{\om^\perp\cap \th^\perp} f(\om, t\th +x)\, dx, \qq (\th, t)\in \tilde\msH, \ee
(set $\eta=\th^\perp$, $v=t\th$, $\xi=\{\om\}$). Here $d_\th \om$ is the probability measure on $\bbs^{n-1}\cap \th^\perp$ that is invariant under orthogonal transformations leaving $\th $ fixed. Similarly, (\ref{dut}) yields
\be\label{kuku02}
(R^*\vp)(\om, u)=\intl_{\bbs^{n-1}\cap \om^\perp}\vp(\th, \th \cdot u)\,  d_\om\th, \qq (\om, u) \in \tilde\msL,\ee
(use the equality $\Pr_{\{\th\}} u=\th \th^Tu$, where $\th$ is interpreted as a matrix with one column and $\th^T$ is its transpose).
By Proposition \ref{prop1} and equalities (\ref{kuk}) and (\ref{kuk1}),
\be\label{du5}
\intl_{\tilde\msH} (Rf)(\th, t)\, \vp(\th, t)\,d_*\th dt=\intl_{\tilde\msL}  f(\om, u) \,(R^*\vp)(\om, u)\, d_*\om du\ee
provided that at least one of these integrals exists in the Lebesgue sense.

\section{Inversion of the Radon Transform of Quasi-Radial Functions}

As we mentioned in Introduction, the Radon transform that takes functions on $\msL$ to functions on $\msH$ is noninjective  because $\dim \msL> \dim \msH$. To remedy the situation, we need  to reduce the dimension of the source space or  increase the dimension  of  the target space. Of course, the geometrical meaning of the problem  should remain unchanged.

Below we  pursue the first approach. Suppose that the value of $f$ at $\ell=\ell (\om, u)$  depends only on $\om$ and $|u|$, that is,  $f(\om, \cdot)$ is constant on the set of all lines equidistant from the central line $\{\om\}$. We call such a set a {\it line  cluster} and denote
 \be \label {nsfo} \text{\rm cl} (\om,r)=\{\ell (\om, u)\in \msL:\, |u|=r\},  \qquad \om\in\bbs^{n-1}, \quad r>0.\ee
 A line function $f$ is called {\it quasi-radial} if it is constant on  all clusters (\ref{nsfo}), that is, $f(\om, u)=f_0(\om, |u|)$ for some function $f_0$ and all or almost all $(\om, u)\in\tilde\msL$. A more restrictive, {\it radial} case, when $f(\om, u)=f_0(|u|)$, that is, $f(\ell)$ depends only on the distance from the origin to $\ell$,  was studied in \cite{Ru04}; cf. (\ref {lmos1}).
The set of all clusters (\ref{nsfo}) has the same dimension $n$, as the set $\msH$ of all hyperplanes. Hence it is natural to expect that
 the restriction of the Radon transform (\ref{kuku01}) to quasi-radial functions   is injective.

\begin{lemma}\label {asi} If  $f$ is quasi-radial,   $f(\om, u)=f_0(\om, |u|)$, then
\be\label{kuku03} (R f)(\th,t)=\sig_{n-3}\intl_{|t|}^{\infty}(r^2-t^2)^{(n-4)/2} \,rdr\intl_{\bbs^{n-1}\cap \th^\perp} f_0(\om, r)\,d_\th \om\,
\ee
provided that  the integral on the right-hand side exists in the Lebesgue sense.
\end{lemma}
\begin{proof} Passing to polar coordinates in (\ref{kuku01}), we obtain
\[
(Rf)(\th, t)=\sig_{n-3} \intl_{\bbs^{n-1}\cap \th^\perp} d_\th \om\intl_0^\infty f_0(\om, \sqrt{s^2+t^2})\,s^{n-3}\,ds.\]
This coincides with (\ref{kuku03}).
\end{proof}

The formula (\ref{kuku03}) is a generalization of  (\ref{lmos1})  for  $k=1$ and $k'=n-1$. It shows that on
quasi-radial functions, $(R f)(\th,t)$ is a constant multiple of the tensor product of the classical Funk transform
\be\label{Funk.aa}
 (F\psi)(\th)=\intl_{\bbs^{n-1}\cap \th^\perp}  \psi(\om) \,d_\th \om, \qquad \th \in \bbs^{n-1},
\ee
and the  fractional integration operator of the Erd\'elyi-Kober type
\be\label{eci} (I^\a_{-, 2} \chi)(t)=\frac{2}{\Gam
(\a)}\intl_t^\infty (r^2 - t^2)^{\a -1} \chi (r)  \, r\, dr,\qquad
t>0,\ee
with $\a= n/2-1$. Both operators were  studied systematically in \cite [Sections 5.1 and 2.6.2] {Ru15}. Thus, for $t>0$, we can write
\be\label{kuku04}
(Rf)(\th, t)= \pi^{n/2-1}   \,(\tilde R f_0)(\th, t),\qq \tilde R=  I^{n/2-1}_{-, 2} \otimes F,\ee
where $F $ acts in the $\om$-variable and $ I^{n/2-1}_{-, 2}$ in the $r$-variable, as in  (\ref{kuku03}).

By Lemma 2.42 from \cite[p. 65]{Ru15} and the existence results for the Funk transform (see, e.g.,  \cite[p. 281]{Ru15}, the integral (\ref{kuku03}) is absolutely convergent for almost all $(\th, t) \in \tilde\msH$ provided that
\be\label{kuku05} \intl_{\bbs^{n-1}} d\om\intl_{a}^{\infty}|f_0(\om, r)|  \,r^{n-3} \,dr<\infty  \q \forall \,a>0,\ee
where the exponent $n-3$ in (\ref{kuku05}) is exact.

A variety of inversion formulas for $F$ and $ I^\a_{-, 2}$ can be found in \cite[subsections 5.1.6-5.1.8, 2.6.2]{Ru15}. For example, the following statement is an immediate consequence of the formulas (2.6.23), (2.6.25), and (5.1.96) from \cite{Ru15}.
\begin{theorem} \label{kuku05a} Let $\vp=Rf$, $f(\om, u)=f_0(\om, |u|)$, where $f_0$ satisfies (\ref{kuku05}). Then
\be\label{kuku06}
f_0(\om, r)= \pi^{1-n/2}   \,(\tilde R^{-1} \vp)(\om, r),\qq \tilde R^{-1}=  \Cal D^{n/2-1}_{-, 2} \otimes F^{-1}.\ee
The Erd\'elyi-Kober  derivative $\Cal D^{n/2-1}_{-, 2}$ is defined by the following formulas:
\be\label{frr+z3a}
\Cal D^{n/2-1}_{-, 2} \chi=(- D)^{n/2-1} \chi , \quad D=\frac {1}{2r}\,\frac {d}{dr},\ee
if $n$ is even, and
 \be\label{frr+z3}  \Cal D^{n/2-1}_{-, 2} \chi=  r\,(- D)^{(n-1)/2} r^{n-2}I^{1/2}_{-,2} \,r^{1-n}\,\chi,\ee
if $n$ is odd, where the powers of $r$ stand for the corresponding multiplication operators. Furthermore,
\be\label{90ashel}
(F^{-1}\psi)(\om) \!=  \! \lim\limits_{t\to 1}  \left (\frac {1}{2t}\,\frac {\partial}{\partial t}\right )^{n-2} \!\left [\frac{2}{(n-3)!}\intl_0^t
(t^2 \!- \!s^2)^{n/2-2} \,\Phi_\om (s) \,s^{n-2}\,ds\right ],\qquad\qquad\ee
\[\Phi_\om (s)=\intl_{\bbs^{n-1}\cap
\om ^\perp}
\!\!\psi(s\, \xi +\sqrt{1-s^2}\,\om)\,d_\om\xi, \qquad -1\le s\le 1.\]
The  limit in (\ref{90ashel}) is understood in the $L^1(\bbs^{n-1})$-norm.
\end{theorem}

\begin{remark}\label{buku02b} {\rm If the function $f$ in Theorem \ref{kuku05a} is smooth, then (\ref{90ashel}) can be replaced by the corresponding expression in terms of the Beltrami-Laplace operator on $\bbs^{n-1}$; see \cite[Theorem 5.37]{Ru04}.} If $f$ is a radial function, i.e., $f(\om, u)=f_0(|u|)$, the spherical component $F$ in (\ref{kuku04}) and (\ref{kuku06}) disappears and we simply have $\tilde R^{-1}=  \Cal D^{n/2-1}_{-, 2}$.  The last formula agrees with \cite[Lemma 2.3]{Ru04}.
\end{remark}

\section{Inversion of the Dual Transform }

\subsection{The Dual Transform of Even Functions}

In this subsection  we confine to hyperplane functions $\vp$ with the property
\be\label{bse}\vp (h)=\vp (-h)  \quad \forall h\in\msH.\ee  A pair $(h, -h)$ of hyperplanes is a natural counterpart of the line cluster  (\ref{nsfo}). It is convenient to take  the dual transform $R^*$ in the form  (\ref{kuku02}), namely,
\be\label{kuku02b}
(R^*\vp)(\om, u)=\intl_{\bbs^{n-1}\cap \om^\perp}\vp(\th, \th \cdot u)\,  d_\om\th,\ee
where $\om \in \bbs^{n-1}$ and $u\in \om^\perp$. In this notation, the function $\vp$ on $\msH$ is identified with a function on the cylinder $\tilde\msH=\{(\th, t):\, \th\in  \bbs^{n-1}, t\in \bbr\}$ satisfying \be\label{ukuy}\vp(\th,t)=\vp(-\th, -t)\quad \forall \; (\th,t)\in \tilde\msH.\ee
Moreover, if $\vp$ is even in the sense of (\ref{bse}), we additionally have
\be\label{ok1}\vp(\th,t)=\vp(\th, -t)=\vp(-\th,t)\quad \forall \; (\th,t)\in \tilde\msH.\ee

\begin{remark} We pay attention to the following interesting fact.
Unlike the Radon transform $R$  that takes quasi-radial functions  on $\msL$ to even functions on $\msH$  (cf. Lemma \ref{asi} and (\ref{ok1})),
 the dual  transform $R^*$ of an  even function  on $\msH$ is not necessarily a quasi-radial function  on $\msL$. Let, for example,
 $n=3, k=1$, and choose  $\vp(\th,t)=|t\th_2|$, where $\th=(\th_1,\th_2,\th_3)\in \bbs^2$ and $t\in\bbr$. Obviously, $\vp$ satisfies (\ref{ok1}).
By (\ref{kuku02b}),
\[(R^*\vp)(e_1, e_2)=\frac{1}{2\pi} \intl_{\bbs^2\cap e_1^{\perp}}\vp(\th,\,\th\cdot e_2)\,d\th.\]
To evaluate this integral, we set \[ \th =\g(\a)\,e_2, \qquad \g(\a)=\begin{bmatrix}1&0&0\\0&\cos\a & \sin\a \\ 0 & -\sin\a &\cos\a \end{bmatrix}, \qquad 0\le \a\le 2\pi.\]
Then
\bea(R^*\vp)(e_1, e_2)&=&\frac{1}{2\pi} \intl_{0}^{2\pi}\vp\left (\begin{bmatrix} 0 \\ \cos\a \\-\sin\a\end{bmatrix}, \cos\a \right)\,d\a\nonumber\\
&=&\frac{1}{2\pi}\intl_{0}^{2\pi}\cos^2\a\,d\a=\frac{1}{2}.\nonumber\eea
Similarly,
\[
(R^*\vp)(e_1, e_3)=\frac{1}{2\pi}\intl_{0}^{2\pi}|\sin\a\cos\a|\,d\a=\frac{1}{\pi}\,.
\]
Thus $(R^*\vp )(e_1, e_2)\neq (R^*\vp)(e_1, e_3)$ which means that $R^*\vp$ is not quasi-radial.
\end{remark}

Now we proceed to reconstruction of $\vp (h)\equiv\vp(\th,t)$ from $(R^*\vp)(\om, u)$ assuming that $\vp$ enjoys the symmetry (\ref{bse})   (or (\ref{ok1})).
Suppose  $u\neq 0$ and apply the slice integration formula from \cite[formula (A.11.12)]{Ru15} to the cross-section $\bbs^{n-1}\cap \om^\perp$. Owing to (\ref{ukuy}), we can write (\ref{kuku02b}) as
\be\label{uku}
(R^*\vp)(\om, u)=c\intl_{0}^{1} \tilde d t \!\!
\intl_{(\bbs^{n-1}\cap\om^{\perp})\cap\tilde{u}^{\perp}} \!\!\!\!
\vp(\sqrt{1-t^2}\sig+t\tilde{u}, tr)\, d_*\sig,\ee
where
\[  c= \frac{2\,\sig_{n-3}}{\sig_{n-2}},\qq\tilde d t=(1-t^2)^{(n-4)/2}dt, \qq r=|u|, \qq \tilde{u}=u/|u|,\]
and $d_*\sig$ denotes the probability measure on the $(n-3)$-dimensional sphere $(\bbs^{n-1}\cap\om^{\perp})\cap\tilde{u}^{\perp}$.

Let $O(\om^\perp) \subset O(n)$ be the stationary subgroup of   the vector $\om$.
 For a function  $f=f(\om,u)$ on $\tilde \msL$ (i.e., $\om\in \bbs^{n-1}$, $u\in\om^{\perp}$), we introduce the mean value operator
 $$
 (\M_{\om}f)(r)=\intl_{O(\om^\perp)}f(\om, r\rho \tilde u)\, d\rho,\qq r>0,$$
that averages $f\equiv f(\ell)$ over the cluster $\text{\rm cl}(\om,r)$.

\begin{lemma} If $\vp$ satisfies (\ref{ok1}), then, for $\om\in \bbs^{n-1}$, $\,r>0$,  and $c=2\,\sig_{n-3}/\sig_{n-2}$,
\be\label{uku1}
(\M_{\om}R^*\vp)(r)=\frac{c}{r^{n-3}}\intl_{0}^{r}(r^2-t^2)^{(n-4)/2}\,dt\!\!\!\intl_{\bbs^{n-1}\cap\om^{\perp}}\!\!\!\vp(\th,t)\, d\th
\ee
provided that the integral on the right-hand side exists in the Lebesgue sense.
\end{lemma}
\begin{proof} Changing the order of integration and using (\ref{uku}), we obtain
\bea (\M_{\om}R^*\vp)(r)&=&\intl_{O(\om^\perp)} (R^*\vp)(\om,r\rho \tilde u)\, d\rho\nonumber\\
&=&c\intl_{0}^{1} \tilde d t\intl_{O(\om^\perp)}d\rho \intl_{(\bbs^{n-1}\cap\om^{\perp})\cap (\rho\tilde u)^{\perp}} \!\!\!\!
\vp(\sqrt{1-t^2}\sig+t\rho\tilde{u}, tr)\, d_*\sig.\nonumber\eea
The latter can be transformed as follows.
\bea (\M_{\om}R^*\vp)(r)&=&c\intl_{0}^{1} \tilde d t\intl_{O(\om^\perp)}d\rho \intl_{(\bbs^{n-1}\cap\om^{\perp})\cap \tilde u^{\perp}} \!\!\!\!
\vp(\sqrt{1-t^2}\rho\th +t\rho\tilde{u}, tr)\, d_*\th\nonumber\\
&=&c\intl_{0}^{1} \tilde d t \intl_{(\bbs^{n-1}\cap\om^{\perp})\cap \tilde u^{\perp}}d_*\th\intl_{O(\om^\perp)}\vp(\rho(\sqrt{1-t^2}\th +t\tilde{u}), tr)\,d\rho\nonumber\\
&=&c\intl_{0}^{1} \tilde d t \intl_{\bbs^{n-1}\cap\om^{\perp}}\vp(\tilde\th, tr)\,d_*\tilde\th\nonumber\\
&=&\frac{c}{r^{n-3}}\intl_{0}^{r}(r^2-t^2)^{(n-4)/2}dt\intl_{\bbs^{n-1}\cap\om^{\perp}}\vp(\tilde\th,t)\,d_*\tilde\th.\nonumber
\eea
This gives (\ref{uku1}).
\end{proof}

The right-hand side of (\ref{uku1}) is a constant multiple of the tensor product (up to weight factors) of the  Funk transform (\ref{Funk.aa})
 and the left-sided modification  of the Erd\'elyi-Kober operator   (\ref{eci}). Specifically,
let
\be\label {as34b12} (I^\a_{+, 2} \chi)(t)=\frac{2}{\Gam
(\a)}\intl_0^r (r^2 -t^2)^{\a -1}\chi (t) \, t\, dt,\qq \a>0;\ee
cf.  \cite [formula (2.6.8)] {Ru15}. Then (\ref{uku1}) can be written as
\[
(\M_{\om}R^*\vp)(r)= c_1\,r^{3-n}\, ([I^{n/2-1}_{+, 2} t^{-1}\otimes F]\,\vp)(\om,r),\qq c_1= \frac{\Gam (n-1)/2)}{\pi^{1/2}} .\]
Here  $F $ acts in the $\th$-argument of $\vp$ and $ I^{n/2-1}_{+, 2}$ acts in the $t$-argument, as in (\ref{uku1}). Note that, owing to (\ref{ok1}), the operator $F$ is injective.

 The integral $(I^\a_{+, 2} \chi)(t)$ is absolutely convergent for almost all $r>0$ whenever $t\chi(t)$ is a locally integrable function on $\bbr_+$; see \cite [p. 65] {Ru15}. It follows that (\ref{uku1}) holds for every locally integrable even function $\vp$ on $\tilde\msH$.
 The inversion procedure for the operator  $I^\a_{+, 2}$ is described in \cite [p. 67] {Ru15}.  In particular,
 \be\label{kur1}
\Cal D^{n/2-1}_{+, 2} \chi=D^{n/2-1} \chi , \qquad D=\frac {1}{2r}\,\frac {d}{dr},\ee
if $n$ is even, and
 \be\label{kur2}  \Cal D^{n/2-1}_{+, 2} \chi=  D^{(n-1)/2}I^{1/2}_{+,2} \chi,\ee
if $n$ is odd.

  Thus an even hyperplane function $\vp$ can be explicitly reconstructed from the line function $R^*\vp$ as follows.

\begin{theorem} \label{kut} Let $\vp$ be a locally integrable  function on $\tilde\msH$ satisfying  (\ref{ok1}), and let
\[ \Phi (\om,r)= \frac{r^{n-3}\pi^{1/2}}{\Gam (n-1)/2)} \, (\M_{\om}R^*\vp)(r), \qq \om\in \bbs^{n-1}, \quad r>0.\]
Then  for $t>0$,
 \be\label{kut1}
\vp(\th, t)=([F^{-1}  \otimes  t \Cal D^{n/2-1}_{+, 2}]\Phi)(\th,t),\ee
where  the inverse Funk transform $F^{-1} $ acts in the $\om$-argument of $\Phi$ by the formula (\ref{90ashel}) and the Erd\'elyi-Kober derivative $\Cal D^{n/2-1}_{+, 2}$, defined by (\ref{kur1})-(\ref{kur2}),  acts in the $r$-argument of $\Phi$.
\end{theorem}

\begin{remark} {\rm As in Remark \ref{buku02b}, if $\vp$  is smooth, then (\ref{90ashel}) can be replaced by the corresponding expression in terms of the Beltrami-Laplace operator. If $\vp$ is  radial, i.e., $\vp(\th, t)=\vp_0(|t|)$, then the spherical component $F^{-1}$ in (\ref{kut1}) disappears. }
\end{remark}

\subsection{The General Case. The Kelvin-Type Transform}

It is known \cite[Section 5]{Ru04} that the Radon transform for a pair of affine Grassmannians and the corresponding dual transform can be expressed one through another by making use of a certain Kelvin-type transformation.
Below we recall this construction.

Let $\t \in \agnk$ be a $k$-plane in $\rn$ not passing
through the origin and parameterized by $\t =\t(\xi, u)$, where $ \xi \in
\gnk, \; u \in \xi^\perp, \; u \neq 0$.
Let $\{\xi, u\}=\span (\xi, u)  \in G_{n,
k+1}$ be the  smallest linear subspace
containing $\t$ and let  $\{\xi, u\}^{\perp} \in G_{n, n-k-1}$ be the
orthogonal complement of $\{\xi, u\}$ in $\rn$. To every $k$-plane $\t =\t(\xi, u)$ with $ u \neq 0$ we associate an
$(n-k-1)$-dimensional plane $\tilde \t=\tilde \t(\tilde \xi, \tilde u)$ defined by
\be\label{ota1}
\tilde \t=\tilde \t(\tilde \xi, \tilde u), \quad \text{\rm where}\quad
\tilde \xi=\{\xi, u\}^{\perp}, \qquad \tilde u=  -\frac{u}{|u|^2},\ee
so that $\tilde u \in {\tilde \xi}^{\perp}$ and $|\tilde u|=|u|^{-1}$ or $|\tilde \t|=|\t|^{-1}$  (we recall that $|\t|$ denotes the Euclidean distance from the origin to the plane $\t$).

The map $\nu :  \t \to \tilde \t$  is called a {\it quasi-orthogonal
inversion map}. It can be regarded as  a Grassmannian modification of the Kelvin transform $x\to x/|x|^2$ for $x \in \rn \setminus \{0\}$.

One can readily see that $\nu$ is an involution, i.e. $\nu=\nu^{-1}$.

By analogy with (\ref{ota1}),  $\nu$  acts from $\agnkp $ to $G(n, n-k'-1)$, so that if $\z=\z(\eta, v)$ with  $v\neq 0$, then $\nu (\z)=\tilde \z$, where
\be\label{ota2}
\tilde \z\equiv \tilde \z(\tilde \eta, \tilde v), \qquad
\tilde \eta=\{\eta, v\}^{\perp}, \qquad \tilde v=  -\frac{v}{|v|^2}.\ee

\begin{example}\label{exa1} If $k=0$ and $\t=x \in \rn \setminus \{0\}$, then $\nu (x)$
 is a hyperplane orthogonal to the vector $x$ and passing through
the point $-x/|x|^2$.
\end{example}

\begin{example}\label{exa2} If $k=1$ and $\t=\ell$ is a line not passing through the origin, that is, $\ell=\ell(\om, u)=\{\om\}+u$, $\om \in \bbs^{n-1}$, $u\in \om^\perp \setminus \{0\}$, then $\nu (\ell)$ is an $(n-2)$-dimensional plane
\[ \pi=\pi (\{\om, u\}^\perp, -u/|u|^2)=\{\om, u\}^\perp -u/|u|^2.\]
\end{example}

\begin{example}\label{exa2a} If $k'=n-1$ and $\z=h(\th,t)=\{x\in \rn: x\cdot \th =t\}$ is a hyperplane in $\rn$ with $t\neq 0$, then $\tilde \z=-\th/t$ is a point in $\rn$.
\end{example}

\begin{definition} Let  $R: f(\t) \to (Rf)(\z)$ be the Radon
transform taking functions on $\agnk$ to functions on $\agnkp, \; k'>k$. If $\tilde \t =\nu (\t)$, and $\tilde \z =\nu (\z)$, then the associated  Radon
transform $\tilde R :  \tilde f(\tilde \z) \to (\tilde R \tilde f)(\tilde\t)$ from functions on
$G(n, n-k'-1)$ to functions on $G(n, n-k-1)$ is called {\it
quasi-orthogonal} to $R$.
\end{definition}

\begin{theorem}\label{exath}   \cite[Theorem 5.5]{Ru04}  Let $0\le k<k' <n$. For a function $\vp$   on $ \agnkp$, we denote
\be\label {hat} (A\vp)(\tilde \z)=|\tilde \z|^{k-n}\vp
(\nu^{-1}(\tilde \z)), \qquad \tilde \z \in G(n, n-k'-1). \ee

{\rm (i)} The following relation holds
\be\label {ild} \intl_{\agnkp}\frac{\vp (\z)
\; d\z}{(1+|\z|^2)^{(k
  +1)/2}}= \frac{\sig_{n-k' -1}}{\sig_{k'}}\intl_{G(n, n-k'-1)}\frac{(A\vp)(\tilde \z)
\; d\tilde \z}{(1+|\tilde \z|^2)^{(k
  +1)/2}} \ee
provided that either side of this equality exists in the Lebesgue sense.

{\rm (ii)} If at least one of the integrals in (\ref{ild}) is finite, then
 \be (R^* \vp)(\t)=c \, |\t|^{k'
-n}(\tilde R A\vp)(\nu (\t)), \qquad c=\frac{\sig_{n-k'
-1}}{\sig_{n-k -1}}. \ee
\end{theorem}

Let us consider the line-to-hyperplane case $1=k<k'=n-1$ in more detail. By Examples \ref{exa2} and \ref{exa2a},
\be\label {hat12} (R^* \vp)(\om, u)=\frac{2}{|u|\,\sig_{n-2}} \,(\tilde R A\vp)(\pi), \quad \pi=\pi (\{\om, u\}^\perp, -u/|u|^2), \ee
where
\be\label {hat1} (A\vp)(x)=|x|^{1-n}\vp (\nu^{-1}(x))\ee
and $\tilde R$ is the Radon-John $(n-2)$-plane transform in $\rn$ \cite {GGG, H11, Ru04a}.
We write (\ref{hat12}) as
\be\label {hat12a} (\tilde R A\vp)(\pi)=\frac{\sig_{n-2}}{2|\pi|} \,(R^* \vp)(\nu^{-1} (\pi)), \quad \pi \in G(n, n-2), \quad 0 \notin \pi.\ee
Inverting  $\tilde R$ by one of the known inversion methods (see, e.g., \cite{H11, Ru04a, Ru13a})
and using (\ref{hat1}), we formally obtain
\be\label {hat5} \vp (h)=|h|^{1-n}  (\tilde R^{-1} \Phi)(\nu (h)), \qquad  h\in \msH,\ee
where
\be\label {hat6} \Phi(\pi)=\frac{\sig_{n-2}}{2|\pi|} \,(R^* \vp)(\nu^{-1} (\pi)).\ee
The equality
\be \label {hat7}\intl_{\msH}\frac{\vp (h)
\; dh}{1+|h|^2}= \frac{2}{\sig_{n-1}}\intl_{\rn}\frac{(A\vp)(x)
\; dx}{1+|x|^2} \ee
which is a particular case of (\ref{ild}), allows us to choose the suitable class of  functions $\vp$. Indeed, by (\ref{hat7}) and
(\ref{hat12a}), the invertibility of $R^* $ on functions $\vp: \msH \to \bbc$ satisfying
\be \label {hat77}\intl_{\msH}\frac{|\vp (h)|
\; dh}{1+|h|^2} <\infty\ee
is equivalent to the invertibility of $\tilde R$ on functions $\tilde\vp: \rn \to \bbc$ satisfying
\be\label {hat78}
\intl_{\rn}\frac{|\tilde\vp(x)|
\; dx}{1+|x|^2} <\infty. \ee
We also  note that by
Theorem 3.2 from \cite{Ru13a}, the condition (\ref{hat78}) guarantees the existence of $(R \tilde\vp)(\pi) $ for almost all $\pi$  (this condition is necessary on nonnegative radial functions $\tilde\vp$).

For the sake of simplicity, we restrict our consideration to the following two subclasses of hyperplane functions satisfying (\ref{hat77}).

For $1\le p<\infty$, we denote
\be\label {mmat1} \tilde L^p (\msH) =\left \{ \vp:\, \intl_{\msH} |h|^{(n-1)(p-1)-2} |\vp (h)|^p \, dh<\infty \right \}.\ee
For $\mu>0$, let  $\tilde C_\mu (\msH)$ be the space of all functions $\vp$ which are
continuous  on the set of all hyperplanes not passing through the origin and satisfy the following condition:
\be\label{dit}
\left\{ \!
 \begin{array} {ll} |h|^{n-1-\mu}\vp(h)=O(1) & \mbox{\rm if  $|h|\to 0$,}\\
 ${}$\\
  |h|^{n-1}\vp(h) \to c=\const   & \mbox{\rm if  $|h|\to \infty$.}\\
   \end{array}
\right.\ee
The space $\tilde C_\mu (\msH)$ is an antipodal modification of the space $C_\mu (\msH)=\{f\in C(\msH):\, f(\t)= O(|\t|^{-\mu})$.
The reason for the above definitions is explained by the following lemma.

\begin{lemma}\label{atu}${}$ \hfill

{\rm (i)} If   $\vp \in \tilde L^p (\msH)$ with    $1\leq p<n/(n-2)$,
then $\vp$ satisfies  (\ref{hat77}). The relations $\vp\in \tilde L^p (\msH) $ and $A\vp\in L^p(\rn)$ are equivalent.

{\rm (ii)}  If   $\vp \in \tilde C_\mu (\msH)$ with    $\mu> n-2$, then $\vp$ satisfies  (\ref{hat77}). The relations  $\vp \in \tilde C_\mu (\msH)$ and $A\vp \in C_\mu(\rn)$ are equivalent.
\end{lemma}

\begin{proof}  (i) The first statement follows by H\"older's inequality. To prove the second statement, we observe that
\[ ||A\vp||_p^p=\intl_{\rn} |x|^{(1-n)p}\, |\vp(\nu^{-1}(x))|^p\, dx=\intl_{\rn}\frac{(A\psi)(x)
\; dx}{1+|x|^2}, \]
where
\[ (A\psi)(x)=|x|^{1-n}\psi(\nu^{-1}(x)), \qquad \psi(h)= |h|^{(n-1)(p-1)}  |\vp(h)|^p  (1+1/|h|^2).\]
Hence, by (\ref{hat7}),
\[
||A\vp||_p^p=\frac{\sig_{n-1}}{2} \intl_{\msH}\frac{\psi (h) \; dh}{1+|h|^2}=\frac{\sig_{n-1}}{2}\intl_{\msH} |h|^{(n-1)(p-1)-2} |\vp (h)|^p \, dh,\]
as desired.

{\rm (ii)} Both statements can be easily checked straightforward.

\end{proof}

Now we can  formulate the  inversion result for $R^*$. Following \cite[Section 3]{Ru13a},  we introduce the mean value operator
\be\label {ular} (\tilde R^*_x \Phi)(r)= \intl_{SO(n)} \!  \Phi (\gam \bbr^{n-2} +x + r\gam e_n) \,
 d\gam, \qquad r>0. \ee

\begin{theorem}\label{ci} ${}$\hfill

{\rm (i)} The function   $\vp \in \tilde L^p (\msH)$,    $1\leq p<n/(n-2)$,
 can be reconstructed from $(R^* \vp)(\ell)$, $\ell \in \msL$,  by the formula (\ref{hat5}) in which $\tilde R^{-1}$ denotes the inverse Radon-John $(n-2)$-plane transform in $\rn$ that can be represented in different forms. For example, Theorem 3.5 from \cite{Ru13a} yields
\be\label{nnxxzz}(\tilde R^{-1}\Phi)(x) =  \lim\limits_{r\to 0}\, \pi^{1-n/2} (\Cal D^{n/2-1}_{-, 2} \tilde R^*_x \Phi)(r),\ee
where the  limit  is understood in the $L^p$-norm and the Erd\'elyi-Kober derivative $\Cal D^{n/2-1}_{-, 2} $ is defined by (\ref{frr+z3a}) and (\ref{frr+z3}).

{\rm (ii)}  If   $\vp \in \tilde C_\mu (\msH)$ with    $\mu> n-2$, then  (\ref{nnxxzz}) holds with the limit interpreted in the $\sup$-norm.
\end{theorem}

This theorem follows immediately from (\ref{hat12a}) and Lemma \ref{atu} if we apply
 Theorems 3.5 and 3.4 from \cite{Ru13a}.

\begin{remark}
The conditions $1\le p<n/(n-2)$ and $\mu> n-2$  in Theorem \ref{ci} are not only of technical nature. In fact, they are necessary for the existence of $R^* \vp$. Let, for instance,
 \be\label{prfs} \vp_p(h)=|h|^{1-n}(2+1/|h|)^{-n/p}\big [\log(2+1/|h|)\big ]^{-1}\,.\ee
One can readily check that this function belongs to  $\tilde L^p (\msH)$ for any $p>1$. However, if $p\ge n/(n-2)$, then $\vp_p$ is not locally integrable and $R^* \vp_p\equiv \infty$. The latter can be easily seen if we apply  (\ref{lmos2}) with $k=1$, $k'=n-1$. The same function with $p=n/\mu$, belonging to $\tilde C_\mu (\msH)$, gives the necessity of the condition $\mu> n-2$.

\end{remark}

  Another inversion formula for $R^* \vp$ can be obtained if we invert   $\tilde R$ by making use of Theorem 5.4 from \cite{Ru04a}. Let
 \bea
 \kappa_\ell  &=&\intl_0^\infty (1- e^{-t})^\ell \,
t^{-n/2} \, dt \nonumber \\
&=&\left\{ \!
 \begin{array} {ll} \Gam (1-n/2) \,\displaystyle{\sum\limits_{j=1}^\ell {\ell \choose j} \,(-1)^j\,
j^{n/2 -1}} & \mbox{\rm if  $n$ is odd,}\\
\displaystyle{ \frac{(-1)^{n/2}}{(n/2 -1) !}  \sum\limits_{j=1}^\ell {\ell \choose j}\,
(-1)^j\, j^{n/2 -1}
\log j} & \mbox{\rm if  $n$ is even.}\\
  \end{array}
\right. \nonumber \eea

\begin{theorem} \label{too1}   Suppose that  $\Phi$ and $\tilde R^*_{x} \Phi$ are defined by (\ref{hat6}) and (\ref{ular}), respectively.
If  $\vp \in \tilde L^p (\msH)$,    $1\leq p<n/(n-2)$, then for any $\ell>n/2 -1$,
 \be\label{pres} \vp (h)=\frac{\pi^{1-n/2}}{\kappa_\ell  \, |h|^{n-1} }
\intl _0^\infty \Big [ \sum_{j=0}^\ell {\ell \choose j} (-1)^j  (\tilde R^*_{\nu (h)} \Phi)(\sqrt {jr})
 \Big ]\, \frac{dr}{r^{n/2}},\ee where
$\int_0^\infty= \lim\limits_{\e \to 0} \int_\e^\infty$ in the
a.e. sense. If $\vp \in \tilde C_\mu (\msH)$,   $\mu> n-2$, this limit is uniform on every compact subset of $\msH$ not containing hyperplanes through the origin.
\end{theorem}

\begin{remark} The right-hand side of (\ref{pres}) resembles
Marchaud's fractional derivative of order $n/2-1$ of the function $(\tilde R^*_{\nu (h)} \Phi)(r)$
 evaluated at $r=0$; cf. \cite[pp. 56, 122]{Ru15}.
 For $n=3$, (\ref{pres})  has
an especially simple form
\be\label{pres1}  \vp (h)=\frac{1}{\pi|h|^2}
 \intl _0^\infty [(\tilde R^*_{\nu (h)} \Phi)(0) -  (\tilde R^*_{\nu (h)} \Phi)(r)]\, \frac{ dr}{r^2}\, . \ee
\end{remark}

\begin{remark}\label{abdx}  It would be  natural to obtain an explicit inversion formula for $R^*\vp$ under the general assumption (\ref{hat77}). This problem is equivalent to inversion of the Radon-John $(n-2)$-plane  transform $\tilde R$ on the space of all functions satisfying (\ref{hat78}).
The last problem can be solved in the sense of distributions, but we cannot give a reference where the desired formula is obtained in the pointwise sense.
 \end{remark}

\begin{remark} {\rm A  function $\vp$ can  be reconstructed from  $R^*\vp$  in a different way if we change the parametrization of the set $\msL$ of all lines.  For example, fix any $x\in \rn$ and set
\be\label{kuu02b}
(R_1^*\vp)(\om, x)=\intl_{\bbs^{n-1}\cap \om^\perp}\vp(\th, \th \cdot x)\,  d_\om\th,\ee
so that
\be\label{kuu}
(R_1^*\vp)(\om, x)=(R^*\vp)(\om, \Pr_{\om^\perp} x).
\ee
For the sake of  simplicity we assume that $\vp$ is smooth. The operator $R_1^*$  averages $\vp$ over all hyperplanes containing the line through the point $x$ in the direction of $\om$.  The function $\tilde \vp_x (\th)=\vp(\th, \th \cdot x)$ is even and can be reconstructed from $(R_1^*\vp)(\om, x)$ using the inverse Funk transform:
\be\label{km02b}\tilde \vp_x (\th)=F^{-1}[(R_1^*\vp)(\cdot, x)](\th)=F^{-1}[(R^*\vp)(\cdot, \Pr_{(\cdot)^\perp} x)](\th).\ee
 It particular, for $x=t\th$, (\ref {km02b}) yields
$\tilde \vp_{t\th} (\th)=\vp(\th, \th \cdot t\th)=\vp(\th, t)$.
This gives the following

\begin{theorem}\label{ffi} An even smooth function $\vp$ on  $\bbs^{n-1} \times \bbr$ can be reconstructed from the dual Radon transform $f(\om, u)=(R^* \vp)(\om, u)$ (see (\ref{kuku02b})) by the formula
\[
\vp(\th, t)=(F^{-1} [f(\cdot, \Pr_{(\cdot)^\perp} x)])(\th)|_{x=t\th}.\]
\end{theorem}

A  simple inversion formula in this theorem  can be explained by the fact that the parametrization  of the set of all lines in (\ref{kuu02b}) is redundant. Indeed, the dimension of the set of all pairs  $(\om, x)$ in (\ref{kuu02b}) is $2n-1$, whereas
the dimension of $\msL=G(n,1)$ is $2n-2$.
}
\end{remark}

\end{document}